\documentclass[conference,10pt,twocolumn]{IEEEtran}

\usepackage[margin=54pt]{geometry}

\IEEEoverridecommandlockouts

\usepackage{mathrsfs}
\usepackage{amsfonts}
\usepackage{amssymb}
\usepackage{amsmath}
\usepackage{graphicx}
\usepackage{caption}
\usepackage{subcaption}
\usepackage{amsthm}
\usepackage{commath}
\usepackage{tikz}
\usepackage{tkz-graph}
\usepackage[europeanresistors]{circuitikz}
\usetikzlibrary{arrows}
\usepackage{flushend}
\usepackage{psfrag}
\usepackage{pgfplots}
\usetikzlibrary{external}
\usepackage{booktabs}
\allowdisplaybreaks

\newtheorem{theorem}{Theorem}

\newtheorem{lemma}[theorem]{Lemma}
\newtheorem{remark}{Remark}

\newtheorem{assumption}{Assumption}
\newtheorem{objective}{Objective}

\DeclareMathOperator*{\argmin}{argmin}

\DeclareMathOperator*{\diag}{diag}

\DeclareMathOperator*{\sgn}{sgn}

\newcommand{\beq}{\begin{equation}}
\newcommand{\eeq}{\end{equation}}
\newcommand{\bq}{\begin{eqnarray}}
\newcommand{\eq}{\end{eqnarray}}
\newcommand{\bqn}{\begin{eqnarray*}}
\newcommand{\eqn}{\end{eqnarray*}}
\newcommand{\bee}{\begin{enumerate}}
\newcommand{\eee}{\end{enumerate}}

\newlength\fheight
\newlength\fwidth

\begin{document}
\title{\vspace*{18pt}  Dynamical Decentralized Voltage Control of Multi-Terminal HVDC Grids}

\author{ \IEEEauthorblockA{Martin Andreasson, Na Li}
 
\thanks{This work was supported in part by the European Commission by  the Swedish Research Council, the Knut and Alice Wallenberg Foundation and NSF CAREER 1553407.}
\thanks{Martin Andreasson is with the ACCESS Linnaeus Centre at KTH Royal Institute of Technology. E-mail: mandreas@kth.se}
	\thanks{ Na Li is with the School of Engineering and Applied Sciences at Harvard University. E-mail: nali@seas.harvard.edu} 
}
\maketitle

\begingroup
\makeatletter
\renewcommand{\p@subfigure}{}
\makeatother

\begin{abstract}
High-voltage direct current (HVDC) is a commonly used technology for long-distance electric power transmission, mainly due to its low resistive losses. When connecting multiple HVDC lines into a multi-terminal HVDC (MTDC) system, several challenges arise. To ensure safe and efficient operation of MTDC systems, the voltage of all terminals need to be steered to within an operational range. In this paper we study the commonly used decentralized voltage droop controller, and show that it in general does not steer the voltages to within the operational range. We propose a  decentralized PI controller with deadband, and show that it always steers the voltages to within the operational range regardless of the loads. Additionally we show that the proposed controller inherits the property of proportional power sharing from the droop controller, provided that both the loads and the line resistances are sufficiently low. The results are validated through simulation in \textsc{Matlab}.
\end{abstract}

\section{Introduction}
Transmitting power over long distances while maintaining low losses is one of the greatest challenges related to power transmission systems. Increased distances between power generation and consumption is a driving factor behind long-distance power transmission. One such example are the large-scale off-shore wind farms, which often require transmitting power over long distances to the mainland power grid. High-voltage direct current (HVDC) power transmission is a commonly used technology for long-distance power transmission. Its higher investment costs compared to AC transmission lines are compensated by its lower resistive losses for sufficiently long distances. The break-even point, i.e., the point where the total costs of overhead HVDC and AC lines are equal, is typically 500-800 km \cite{padiyar1990hvdc}. However, for cables which are typically used in undersea connections, the break-even point is typically lower than 100 km \cite{bresesti2007hvdc}. 
As more energy sources and consumers are connected by HVDC lines, the individual lines will eventually form a grid consisting of multiple terminals connected by several HVDC transmission lines. Such systems are referred to as Multi-terminal HVDC (MTDC) systems in the literature \cite{van2010multi}. 

One of the main control challenges for MTDC transmission systems is to maintain adequate DC voltages at the terminals. Firstly, the voltage levels at the DC terminals govern the current flows by Ohm's law and Kirchhoff's circuit laws. Secondly, if the DC voltage deviates too far from the nominal operational voltage, equipment could be damaged, resulting in loss of power transmission capability \cite{kundur1994power, van2010multi}. An additional, desirable property of MTDC systems is, analogous to AC power systems, the ability to fairly share injected DC currents after a disturbance has occurred \cite{kundur1994power, Haileselassie2012}. This property is often referred to as \emph{power sharing} or \emph{current sharing}. While power sharing is not strictly necessary from a safety perspective, it can reduce the cost of voltage regulation by distributing the injected current optimally. 
 
Different voltage control schemes for HVDC systems have been proposed in the literature. Among them, the voltage droop controller is the most well-known scheme \cite{kundur1994power, haileselassie2009control, Haileselassie2012}.
The voltage droop controller is a proportional controller, regulating current injections based on the local voltage \cite{14, haileselassie2009control}. An inherent disadvantage of proportional control is the presence of static control errors, and the voltage droop controller is no exception. To eliminate static control errors induced by proportional voltage droop controllers, several 
distributed secondary controllers have been proposed for MTDC systems  \cite{dai2010impact, andreasson2014MTDCjournal}. Existing secondary control schemes for MTDC systems are however, with few exceptions, distributed (with local communication) or centralized, in the sense that local controllers need access to remote state information. This can cause problems for MTDC systems, whose dynamics are very fast, when communication delays become an issue. 

In this paper, we focus on developing and analyzing the performance of completely decentralized voltage controllers with respect the two aforementioned objectives. One is to steer the DC voltages into a predefined range and the other is to fairly share the injected DC currents. We first study the standard voltage droop controller and the PI controller (without deadband). We show that the voltage droop controller is not able to achieve the first objective but is able to achieve the second one under certain conditions, whereas the PI controller is able to achieve the first objective but is not able to achieve the second one. Lastly, we propose a completely decentralized dynamical controller, which is based on a PI controller with deadband. We show that the proposed controller achieves the first objective and that under some conditions it achieves the second objective. Since the proposed controller is completely decentralized, it is much more robust to delays than the distributed controllers in the literature. 

The remainder of this paper is organized as follows. In Section~\ref{sec:model} we introduce our model and formally define our control problem. In Section~\ref{sec:droop} we review the classic droop controller. In Section~\ref{sec:dec_PI} we study a decentralized PI controller, before introducing a modified PI controller with deadband in Section~\ref{sec:deadband_PI}. In Section~\ref{sec:simulations} the controller is validated through simulation. The paper ends by some concluding remarks in Section~\ref{sec:discussion}.

\section{Model and problem setup}
\label{sec:model}

\subsection{Notation}
\label{subsec:prel}
Let $\mathcal{G}$ be a graph. Denote by $\mathcal{V}=\{ 1,\hdots, n \}$ the vertex set of $\mathcal{G}$, and by $\mathcal{E}=\{ 1,\hdots, m \}$ the edge set of $\mathcal{G}$. Let $\mathcal{N}_i$ be the set of neighboring vertices to $i \in \mathcal{V}$.
In this paper we will only consider fixed, undirected and connected graphs. For the control of MTDC power transmission systems, this is a reasonable assumption as long as there are no power line failures.
Denote by $\mathcal{B}$ the vertex-edge adjacency matrix of a graph, and let $\mathcal{\mathcal{L}_W}=\mathcal{B}W\mathcal{B}^T$ be its weighted Laplacian matrix, with edge-weights given by the  elements of the positive definite diagonal matrix $W$.
Let $c_{n\times m}$ be a matrix of dimension $n\times m$ whose elements are all equal to $c$, and by $c_n$ a column vector whose elements are all equal to $c$. 
We will often drop the notion of time dependence of variables, i.e., $x(t)$ will be denoted $x$ for simplicity.

\subsection{Model}
Consider an MTDC transmission system consisting of $n$ HVDC terminals. The terminals are denoted by the vertex set $ \mathcal{V}= \{1, \dots, n\}$. The DC terminals are connected by $m$ HVDC transmission lines, denoted by the edge set $ \mathcal{E}= \{1, \dots, m\}$. 
The HVDC lines are assumed to be purely resistive, neglecting capacitive and inductive elements of the HVDC lines. The assumption of purely resistive lines is not restrictive for the control applications considered in this paper \cite{kundur1994power}. In fact, the line capacitance can be included in the capacitances of the terminals. This implies that 
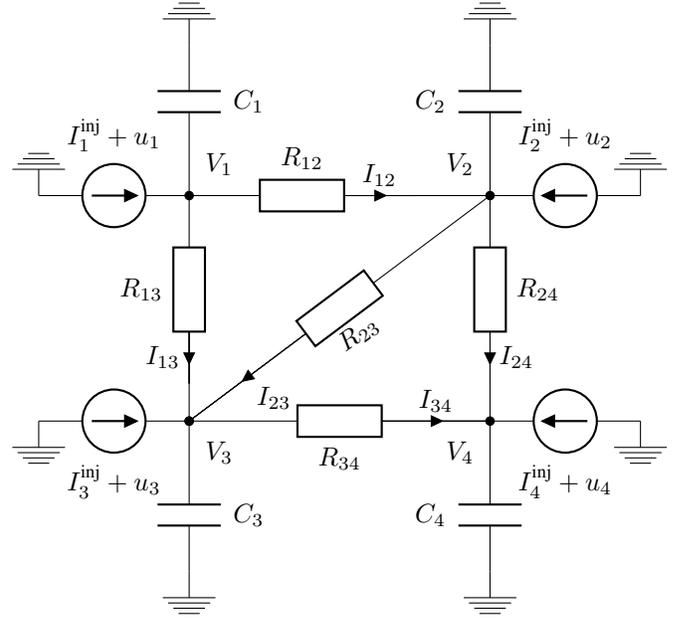
\begin{figure}
\center
\tikzsetnextfilename{HVDC_grid_figure_dynamic}
\begin{tikzpicture}[american voltages]
\draw 
(0,2) node[ground, rotate=180] {}
to 	[C, l=$C_1$] (0,0.5)
to 	[short, -*]  (0,0)
to  [R, l^=$R_{12}$, i^=$I_{12}$] (3,0)
to  [short, -*] (4,0)
to	(4,0.5) [C, l=$C_2$]
to	(4,2) node[ground, rotate=180] {}
(0.4,0.42) node[]  {$V_1$}
(3.6,0.42) node[]  {$V_2$}
[short, l={}]  (0,0)
to  [R, l_=$R_{13}$, i_=$I_{13}$] (0,-2.5)
to  [short, -*] (0,-3)
[short, l={}]  (4,0)
to  [R, l^=$R_{24}$, i^=$I_{24}$] (4,-2.5)
to  [short, -*] (4,-3)
[short, l={}]  (4,0)
to  [R, l=$R_{23}$, i=$I_{23}$] (0,-3)
to  [short, -*] (0,-3)
(0,-5) node[ground] {}
to 	[C, l_=$C_3$] (0,-3.5)
to 	[short, l={}]  (0,-3)
to  [R, l_=$R_{34}$, i^=$I_{34}$] (4,-3)
to  [short, l={}] (4,-3)
to	(4,-3.5) [C, l_=$C_4$]
to	(4,-5) node[ground] {}
(0.4,-3.42) node[]  {$V_3$}
(3.6,-3.42) node[]  {$V_4$}
(-2,0) node[ground, rotate=180] {}
to [american current source, l^=$I_1^\text{inj}+u_1$] (0,0)
to [short, l={}] (0,0)
(6,0) node[ground, rotate=180] {}
to [american current source, l_=$I_2^\text{inj}+u_2$] (4,0)
to [short, l={}] (4,0)
(-2,-3) node[ground] {}
to [american current source, l_=$I_3^\text{inj}+u_3$] (0,-3)
to [short, l={}] (0,-3)
(6,-3) node[ground] {}
to [american current source, l^=$I_4^\text{inj}+u_4$] (4,-3)
to [short, l={}] (4,-3);
\end{tikzpicture}
\caption{Topology of an MTDC system consisting of 4 terminals and 5 lines.}
\label{fig:graph}
\end{figure}
\begin{align*}
I_{ij} = \frac{1}{R_{ij}} (V_i -V_j),
\end{align*}
due to Ohm's law, where $V_i$ is the voltage of terminal $i$, $R_{ij}$ is the resistance and $I_{ij}$ is the current of the HVDC line from terminal $i$ to $j$. 
The voltage dynamics of an arbitrary DC terminal $i$ are assumed to be given by
\begin{align}
C_i \dot{V}_i &= -\sum_{j\in \mathcal{N}_i} I_{ij} + I_i^{\text{inj}} + u_i \nonumber \\
&= -\sum_{j\in \mathcal{N}_i} \frac{1}{R_{ij}}(V_i -V_j) + I_i^{\text{inj}} + u_i,
\label{eq:hvdc_scalar}
\end{align}
where $C_i$ is the total capacitance of terminal $i$, including any shunt capacitances and capacitance of the incident HVDC line, $I_i^{\text{inj}}$ is the uncontrolled injected current, which is assumed to be unknown but constant over time, and $u_i$ is the controlled injected current.  In Figure~\ref{fig:graph}, a four terminal MTDC system is illustrated. 
Equation \eqref{eq:hvdc_scalar} may be written in vector-form as
\begin{align}
\begin{aligned}
\dot{V} &= -E\mathcal{L}_R V +EI^{\text{inj}} + Eu,
\end{aligned}
\label{eq:hvdc_vector}
\end{align}
where $V=[V_1, \dots, V_n]^T$, $E=\diag([C_1^{-1}, \dots, C_n^{-1}])$, $I^{\text{inj}} = [I^{\text{inj}}_1, \dots, I^{\text{inj}}_n]^T$, $u=[u_1, \dots, u_n]^T$ and $\mathcal{L}_R$ is the weighted Laplacian matrix of the graph representing the transmission lines, whose edge-weights are given by the conductances $\frac{1}{R_{ij}}$.

In practice, users and generators at each region $i$ will determine their local injected current $I_i^{\text{inj}}$, which will affect the global system voltage profile $V$. Note that $I^{\text{inj}}$ can be viewed as the disturbance to the system.  The operation of an MTDC system requires the voltages at the terminals to be within a specified range through the controller $u$, otherwise equipment could be damaged. 
Hence we require the terminal voltages to be asymptotically steered into an acceptable operating range, which is formalized in the following objective.  
\begin{objective}
\label{obj:voltage_range}
The HVDC terminal voltages are asymptotically steered into the acceptable operating range under any disturbance $I^{inj}$, i.e.,
\begin{align}
\label{eq:voltage_range}
\lim_{t\rightarrow \infty} V_i(t) \in [\underline{V}_i, \overline{V}_i], \; i \in \mathcal{V}.
\end{align}
\end{objective}

Another desirable property of an MTDC system is its ability to fairly distribute the cost of regulating the terminal voltages in case of an increased load. 
Below we consider an optimization problem with quadratic cost functions on the controlled power injection. 
\begin{objective}
\label{obj:current_sharing_opt}
The cost of the controlled current injections should be minimized asymptotically. More precisely, we require
\begin{align*}
\lim_{t\rightarrow \infty} u(t) = u^*,
\end{align*}
where $u^*$ is defined by
\begin{align}
[u^*,V^*]=\argmin_{[u,V]} \sum_{i \in \mathcal{V}} \frac{1}{2} f_i u_i^2 \quad \text{s.t.} \quad \mathcal{L}_R V  &= I^{\text{inj}} + u, \label{eq:opt1}
\end{align}
and where $f_i>0, i=1, \dots , n$ are positive constants. 
\end{objective}
\begin{remark}
\label{rem:optimality_opjective}
Given a connected graph $\mathcal{G}$, the Laplacian matrix $\mathcal{L}_R$ has rank $n-1$, and $1_n$ is the only eigenvector corresponding to the zero eigenvalue. This implies that the optimization problem \eqref{eq:opt1} defines the optimal voltages $V^*$ of the terminals only up to the addition of a scalar constant. And the total current injection $I^{inj}+u^*$ should be in the image space of $\mathcal{L}_R$. 
\end{remark}

\begin{lemma}
 \label{lemma:optimality}
Objective \ref{obj:current_sharing_opt} is satisfied if and only if $\lim_{t\rightarrow \infty} u(t) =\mu F^{-1}1_{n}$ and $\lim_{t\rightarrow \infty} \mathcal{L}_R V(t)  = I^{\text{inj}} + \mu F^{-1}1_{n}$, where $F = \diag(f_1, \dots, f_n)$. The scaling factor is given by $\mu = -(\sum_{i=1}^n I_i^{\text{inj}})/(\sum_{i=1}^n f_i^{-1})$. 
 \end{lemma}
\begin{remark}
\label{rem:optimality_lemma}
Lemma~\ref{lemma:optimality} implies that the optimal solution of \eqref{eq:opt1} is given by proportional current sharing between the HVDC terminals, and that the current sharing ratio of the terminals are inversely proportional to the cost coefficients $f_i$. 
\end{remark} 
 
 \begin{proof}
The KKT condition for the optimization problem \eqref{eq:opt1} is $Fu=\mu 1_{n}$, which gives $u= F^{-1}\mu 1_{n}$. Substituting this expression for $u$ and pre-multiplying the constraint $\lim_{t\rightarrow \infty} \mathcal{L}_R V(t)  = I^{\text{inj}} + F^{-1}\lambda 1_{n}$ with $1_n^T$, yields the desired expression for $\mu$. 
Since \eqref{eq:opt1} is convex, the KKT condition is a necessary and sufficient condition for optimality.
 \end{proof}
 
\section{Decentralized droop control}
\label{sec:droop}
In this section we study the commonly used voltage droop controller. The voltage droop controller takes the form of a decentralized proportional controller, and is given by:
\begin{align}
u_i &= - K^P_i(V_i - V^{\text{nom}}_i),
\label{eq:droop}
\end{align}
where $V^{\text{nom}}_i$ is the nominal DC voltage of terminal $i$. Throughout this paper, we assume that $V^{\text{nom}}_i = V^{\text{nom}}_j, \; i,j \in \mathcal{V}$ and let $V^{\text{nom}} = [V^{\text{nom}}_1, \dots, V^{\text{nom}}_n]^T$. Also let $K^P=\diag(K^P_1, \dots, K^P_n)$. The controller \eqref{eq:droop} can be written in vector form as
\begin{align}
u &= - K^P(V-V^{\text{nom}}).
\label{eq:droop_vector}
\end{align}
The decentralized structure of the voltage droop controller is appealing to control of HVDC terminals, as the voltage dynamics are typically very fast, and hence sensitive to controller communication delays. 

The closed-loop dynamics of the MTDC system \eqref{eq:hvdc_vector} and the droop controller \eqref{eq:droop_vector} is given by
\begin{align}
\begin{aligned}
\dot{{V}}' &= -E\mathcal{L}_R V - EK^P(V-V^\text{nom}) + EI^{\text{inj}} \\
&= {-E(\mathcal{L}_R + K^P)}{V}' + EI^{\text{inj}},
\end{aligned}
\label{eq:hvdc_closed_loop_droop_vector}
\end{align}
where $V'=V-V^\text{nom}$.\footnote{To obtain the right hand side of (\ref{eq:hvdc_closed_loop_droop_vector}), notice that $\mathcal{L}_R V^{nom}=0$.} 
It hast been shown in \cite{Andreasson2014_IFAC} that the equilibrium of \eqref{eq:hvdc_closed_loop_droop_vector} is always asymptotically stable, as long as the droop gains $K^P_i$ are positive. 
The droop controller (\ref{eq:droop}) however in general fails to guarantee that the voltages converge within a certain bound. The following lemma formalizes this claim.

\begin{lemma}
\label{lemma:droop_voltage}
An MTDC system \eqref{eq:hvdc_scalar} controlled by the droop controller \eqref{eq:droop} where $K^P$ is fixed, does in general not satisfy Objective~\ref{obj:voltage_range} for any finite $\underline{V}_i$ and $\overline{V}_i$. 
\end{lemma}

\begin{proof}
Consider the equilibrium of \eqref{eq:hvdc_closed_loop_droop_vector}, which gives $ V' = (\mathcal{L}_R + K^P)^{-1}  I^{\text{inj}}$, where $V'=V-V^\text{nom}$. Since $K^P$ is fixed, $(\mathcal{L}_R + K^P)^{-1}$ exists and is finite. If $I^{\text{inj}}$ is chosen as, e.g., any eigenvalue of $(\mathcal{L}_R + K^P)$, the magnitude of $V'$ can be increased arbitrarily by increasing the magnitude of $I^{\text{inj}}$, eventually violating Equation~\eqref{eq:voltage_range} of Objective~\ref{obj:voltage_range}.
\end{proof}

It can easily be shown that an MTDC system controlled with the droop controller \eqref{eq:droop} does in general also not satisfy Objective~\ref{obj:current_sharing_opt}. Under certain special conditions however, Objective~\ref{obj:current_sharing_opt} can be satisfied, as detailed in the following lemma. 
Now let the resistances of the HVDC lines be scaled by a factor $\gamma$, i.e., $\tilde{R}_{ij} = \gamma R_{ij}, \; (i,j) \in\mathcal{E}$. Because HVDC usually features with low power line resistance and low power losses \cite{padiyar1990hvdc}, the following lemma studies the performance of the droop controller (\ref{eq:droop}) when $\gamma$ is approximately zero.

\begin{lemma}
\label{lemma:droop_current_sharing_resistance}
In the limit when the HVDC line resistances $\tilde{R}_{ij}$ go to zero, i.e., $\gamma \rightarrow 0, (i,j) \in \mathcal{E}$, an MTDC system \eqref{eq:hvdc_scalar} controlled by the droop controller \eqref{eq:droop} satisfies Objective~\ref{obj:current_sharing_opt} for $F^{-1}=K^P$. 
\end{lemma}

\begin{proof}
Let $\tilde{\mathcal{L}}_R$ be the weighted Laplacian matrix of the MTDC system, with edge-weights $\frac{1}{\gamma R_{ij}}$. 
Since $(\tilde{\mathcal{L}}_R + K^P)$ is real and symmetric, any vector in $\mathbb{R}^n$ can be expressed as a linear combination of its eigenvectors. Denote by $(v_i, \lambda_i)$ the eigenvector and eigenvalue pair $i$ of $(\tilde{\mathcal{L}}_R + K^P)$. Note that $v_i$ and $\lambda_i$ will depend on $\gamma$, but that we do not write the dependence explicitly to simplify notation. Since $(\tilde{\mathcal{L}}_R + K^P)$ is symmetric, the eigenvectors ${v_i, i=1, \ldots, n}$ can be chosen such that they form an orthonormal basis of $\mathbb{R}^n$. Also we order the eigenvalues $\lambda_i$ such that $\lambda_1\geq \lambda_2 \geq \ldots \lambda_n$. Write
\begin{align}
I^{\text{inj}}  &= \sum_{i=1}^n a_i v_i,
\label{eq:eigendecomposition_1}
\end{align}
where $a_i, i=1, \dots, n$ are real constants. which also depend on $\gamma$. The equilibrium of \eqref{eq:hvdc_closed_loop_droop_vector} implies that the voltages satisfy
\begin{align*}
 V'
&= (\tilde{\mathcal{L}}_R + K^P)^{-1}  I^{\text{inj}}  \\
&=  \left(\frac{1}{\gamma} \mathcal{L}_R + K^P\right)^{-1}\sum_{i=1}^n a_i v_i \\
&= \sum_{i=1}^n \frac{a_i}{\lambda_i} v_i,
\end{align*}
where $\lambda_i$ is the $i$th eigenvalue of $\left(\frac{1}{\gamma} \mathcal{L}_R + K^P\right)$. Let us consider $\left( \mathcal{L}_R + \gamma K^P\right)$ instead, which clearly has the same eigenspace, but whose eigenvalues $\lambda_i'$ are scaled by $\frac{1}{\gamma}$, i.e., $\lambda_i' = {\gamma} \lambda_i$. 
Clearly $\lim_{\gamma \rightarrow 0}\left( \mathcal{L}_R + \gamma K^P\right) = \mathcal{L}_R$, and hence $\lambda_i'$ are the eigenvalues of the Laplacian matrix $\mathcal{L}_R$ as $\gamma \rightarrow 0$. Since the eigenvalues $\lambda_i(\mathcal{L}_R), \; i = 2, \dots n$, of $\mathcal{L}_R$ are all nonzero, $\lim_{\gamma \rightarrow 0} \lambda_i =  \lim_{\gamma \rightarrow 0} \frac{1}{\gamma} \lambda_i' = \infty, \; i=2, \dots, n$. 
Thus $\lim_{\gamma \rightarrow 0} V' = \frac{a_1}{\lambda_1} v_1$ where $v_1$ is eigenvector of $\mathcal{L}_R$ corresponding to the eigenvalue $0$. Since the limit of $V'$ exists and is in general nonzero, $\lambda_1$ must be nonzero and finite. 
By premultiplying \eqref{eq:eigendecomposition_1} with $v_1^T$ and keeping in mind that $v_1 = \frac{1}{\sqrt{n}}1_n$, we obtain $a_1=\frac{1}{\sqrt{n}} \sum_{i=1}^n I_i^{\text{inj}}$ since the eigenvectors of $(\tilde{\mathcal{L}}_R + K^P)$ form an orthonormal basis of $\mathbb{R}^n$. Thus $\lim_{\gamma \rightarrow 0} \lim_{t\rightarrow \infty} V'(t) = \frac{1}{{n}\lambda_1}\left(\sum_{i=1}^n I_i^{\text{inj}}\right) 1_{n}$. Finally the controlled injected currents are given by
\begin{align*}
\lim_{\gamma \rightarrow 0} u = \lim_{\gamma \rightarrow 0} - K^P V'& = K^P\left( \frac{1}{{n}\lambda_1}\left(\sum_{i=1}^n I_i^{\text{inj}}\right) 1_{n} \right) \\
&= -\frac{a_1}{\sqrt{n}\lambda_1} K^P 1_{n}.
\end{align*}
By premultiplying the equilibrium of \eqref{eq:hvdc_closed_loop_droop_vector} with $1_{n}^T C$ we obtain
\begin{align*}
1_n^T K^P V' =  1_n^T I^{\text{inj}},
\end{align*}
which implies that
\begin{align*}
\frac{a_1}{\lambda_1} &= \frac{1_{n}^T I^{\text{inj}}}{1_{n}^T  K^P 1_{n} } = (\sum_{i=1}^n I_i^{\text{inj}})/(\sum_{i=1}^n K^P_i),
\end{align*}
which gives the following expression for $u$:
\begin{align*}
u= -\left(\sum_{i=1}^n I_i^{\text{inj}}\right)/\left(\sum_{i=1}^n K^P_i\right) K^P 1_{n}.
\end{align*}
By Lemma~\ref{lemma:optimality}, Objective~\ref{obj:current_sharing_opt} is satisfied for $F^{-1}=K^P$. 
\end{proof}

\section{Decentralized PI control}
\label{sec:dec_PI}
A na\^{i}ve approach to satisfy Objective~\ref{obj:voltage_range} with a decentralized control structure, is to implement decentralized PI controllers at the HVDC terminals. By adding an integrator at each terminal, any steady-state error can in theory be eliminated. If the reference voltage is chosen to be within the acceptable voltage region, Objective \ref{obj:voltage_range} should be satisfied. 
Based on the previous discussion, the controller at an arbitrary terminal $i$  takes the form
\begin{subequations}
\label{eq:dec_voltage_control_PI}
\begin{align}
u_i &= -K^P_i (V_i - V_i^\text{nom}) - K_i^I z_i \label{subeq:dec_voltage_control_PI_1} \\
\dot{z}_i &= (V_i - V_i^\text{nom}). \label{subeq:dec_voltage_control_PI_2} 
\end{align}
\end{subequations}

\begin{lemma}
\label{lemma:PI_voltage}
An MTDC system \eqref{eq:hvdc_scalar} controlled by the PI controller \eqref{eq:dec_voltage_control_PI} satisfies Objective~\ref{obj:voltage_range} if and only if $\underline{V}_i \le V_i^\text{nom} \le \overline{V}_i, \; i=1 \dots, n$.  
\end{lemma}

\begin{proof}
The stability of the closed-loop dynamics of \eqref{eq:hvdc_closed_loop_droop_vector} and \eqref{eq:dec_voltage_control_PI} follows from Theorem~\ref{th:stability_ss_PI_deadband} in the next section by letting $\underline{V}_i = V_i^\text{nom} = \overline{V}_i, \; i=1 \dots, n$. Letting $\dot{z}_i=0, \; i=1, \dots, n$ gives $V_i = V_i^\text{nom}, \; i=1, \dots, n$, which concludes the proof.
\end{proof}
While the decentralized PI controller always satisfies Objective~\ref{obj:voltage_range}, it does in general not satisfy Objective~\ref{obj:current_sharing_opt}. The following lemma formalizes this claim. 

\begin{lemma}
\label{lemma:PI_current_sharing}
An MTDC system \eqref{eq:hvdc_scalar} controlled by the PI controller \eqref{eq:dec_voltage_control_PI} does in general not satisfy Objective~\ref{obj:current_sharing_opt}.
\end{lemma}

\begin{proof}
The closed-loop dynamics of the MTDC system \eqref{eq:hvdc_vector} and the decentralized PI controller \eqref{eq:dec_voltage_control_PI} is given by
\begin{subequations}
\label{eq:hvdc_dec_voltage_control_PI_cl}
\begin{align}
\dot{V} &= -E \mathcal{L}_R V' - E K^P V' - E K^I z + E I^\text{inj} \label{subeq:hvdc_dec_voltage_control_PI_cl_a} \\
\dot{z} &= V', \label{subeq:hvdc_dec_voltage_control_PI_cl_b}
\end{align}
\end{subequations}
where $V'=V-V^\text{nom}$. Now consider the equilibrium of \eqref{eq:hvdc_dec_voltage_control_PI_cl}, which implies $V'=0_n$ and $z=(K^I)^{-1}I^\text{inj}$. This implies that in steady-state, $u=-K^Iz = -I^\text{inj}$. This implies that any change in load, $I^\text{inj}$, is compensated for locally at the terminal. Clearly Objective~\ref{obj:current_sharing_opt} is not satisfied, since $u\ne \mu F^{-1}1_{n}$ in general. 
\end{proof}

\section{Decentralized PI control with deadband}
\label{sec:deadband_PI}
As we have seen in Section~\ref{sec:dec_PI}, pure decentralized PI controllers at the terminals cannot guarantee current sharing by Objective~\ref{obj:current_sharing_opt}. In fact, the injected currents depend entirely on the disturbance $I^\text{inj}$, and $u=-I^\text{inj}$ at steady-state. This even defeats the purpose of an MTDC system, as any disturbance is compensated for locally, and there are no HVDC currents in steady-state. 
A suitable controller for an MTDC system should possess the flexibility and current sharing property of a droop controller, while guaranteeing that the voltages converge to within the acceptable operating range. 
In this section we propose a decentralized dynamical controller which is guaranteed to satisfy Objective~\ref{obj:voltage_range} and, in some cases, also Objective~\ref{obj:current_sharing_opt}. 
The following decentralized controller is proposed:
\begin{subequations}
\label{eq:dec_voltage_control}
\begin{align}
u_i &= -K^P_i (V_i - V_i^\text{nom}) - K_i^I z_i \label{subeq:dec_voltage_control_1} \\
\dot{z}_i &= [V_i - \overline{V}_i]^+ - [\underline{V}_i - V_i]^+ \label{subeq:dec_voltage_control_2} 
\end{align}
\end{subequations}
where $K^P_i, K^I_i \; i=1, \dots, n$ are positive controller gains, $[x]^+=\max(x,0)$, and $V_i^\text{nom}$ is the reference voltage of terminal $i$. 
We assume that the nominal voltage for every terminal is within the operational range.
\begin{assumption}
\label{ass:nom_voltage}
The nominal voltages satisfy $\underline{V}_i \le V_i^\text{nom} \le \overline{V}_i, \; i\in \mathcal{V}$.  
\end{assumption}
\begin{remark}
Note that if $K^I_i=0, \; i=1, \dots, n$, the controller \eqref{eq:dec_voltage_control} reduces to the droop controller \eqref{eq:droop}. If however $\underline{V}_i = V_i^\text{nom} = \overline{V}_i$, the controller \eqref{eq:dec_voltage_control} reduces to  the decentralized PI controller \eqref{eq:dec_voltage_control_PI}. 
\end{remark}
We now show that the controller \eqref{eq:dec_voltage_control} asymptotically steers the voltages of \eqref{eq:hvdc_scalar} into the acceptable operating range, regardless of the uncontrolled injected currents $I^\text{inj}_i$. Inserting \eqref{eq:dec_voltage_control} in \eqref{eq:hvdc_scalar} and introducing the state vectors $V=[V_1, \dots, V_n]^T$ and $z=[z_1, \dots, z_n]^T$ yields
\begin{subequations}
\label{eq:hvdc_dec_voltage_control_cl}
\begin{align}
\dot{V} &= -E \mathcal{L}_R V - E K^P (V-V^\text{nom}) - E K^I z + E I^\text{inj} \label{subeq:hvdc_dec_voltage_control_cl_a} \\
\dot{z} &= [V - \overline{V}]^+ - [\underline{V} - V]^+, \label{subeq:hvdc_dec_voltage_control_cl_b}
\end{align}
\end{subequations}
where  $E=\diag(C_1^{-1}, \dots, C_n^{-1})$ is the matrix of elastances,  $\overline{V} = [\overline{V}_1, \dots, \overline{V}_n]$, $\underline{V} = [\underline{V}_1, \dots, \underline{V}_n]$ and $K^I=\diag(K^I_1, \dots, K^I_n)$. 
\begin{theorem}
\label{th:stability_ss_PI_deadband}
The dynamics \eqref{eq:hvdc_dec_voltage_control_cl} satisfy Objective~\ref{obj:voltage_range} if the voltage bounds are uniform, i.e., $\overline{V}_i=\overline{V}_j, \underline{V}_i=\underline{V}_j, \;\forall i,j \in \mathcal{V}$.
\end{theorem}

\begin{proof}
Let ${V}'=V-V^\text{nom}$ and $z^0$ be defined as a stationary solution to \eqref{subeq:hvdc_dec_voltage_control_cl_a} where ${V}'=0_n$, i.e.,
\begin{align}
\mathcal{L}_R V^\text{nom}  -K^I z^0 + I^\text{inj}=0. 
\end{align}
Defining $z'=z-z^0$, we rewrite \eqref{eq:hvdc_dec_voltage_control_cl} as 
\begin{subequations}
\label{eq:hvdc_dec_voltage_control_cl_hat}
\begin{align}
\dot{{V}}' &= -E \mathcal{L}_R {V}' - E K^P{V}' - E K^I z' \label{subeq:hvdc_dec_voltage_control_cl_hat_a} \\
\dot{z}' &= [{V}' - {\overline{V}}']^+ - [{\underline{V}}'  - {V}']^+, \label{subeq:hvdc_dec_voltage_control_cl_hat_b}
\end{align}
\end{subequations}
where ${\overline{V}}' = \overline{V} - V^\text{nom}$, ${\underline{V}}' = \underline{V} - V^\text{nom}$.
Consider now the following Lyapunov function candidate:
\begin{IEEEeqnarray*}{lcl}
W({V}',z') & {}={} & \frac 12 \left( [{V}' - {\overline{V}}']^+ - [{\underline{V}}' - {V}']^+ \right)^T C \\
&& \times \left( [{V}' - {\overline{V}}']^+ - [{\underline{V}}' - {V}']^+ \right) + \frac{1}{2} z'^T K^I z',
\end{IEEEeqnarray*}
where $C=\diag(C_1, \dots, C_n)$. We note that ${\overline{V}}' \ge 0_n$ and ${\underline{V}}'  \le 0_n$, where the inequalities are taken component-wise.
 Clearly $W({V}',z')\ge 0$, and the set where $W({V}',z') = 0$ is given by $S_1 = \{ [{V'}^T,z'^T]^T| W({V}',z') = 0 \} = \{ [{V'}^T,z'^T]^T | {\underline{V}}'  \le {V}' \le {\overline{V}}', z'=0 \} $. 
 Differentiating $W({V}',z')$ along trajectories of \eqref{eq:hvdc_dec_voltage_control_cl_hat} we obtain
{\allowdisplaybreaks
\begin{IEEEeqnarray*}{lcl}
\dot{W}({V}',z') & {}={} & \left( [{V}' - {\overline{V}}']^+ - [{\underline{V}} - \hat{V}']^+ \right)^T C \\
&& \times \frac{\text{d}}{\text{d}t} \left( [{V}' - {\overline{V}}']^+ - [{\underline{V}}' - {V}']^+ \right) +  z'^T K^I \dot{z}' \\
& {}={} & -\left( [{V}' - {\overline{V}}']^+ - [{\underline{V}}' - {V}']^+ \right)^T C  \\
&& \times \sgn \left( [{V}' - {\overline{V}}']^+ + [{\underline{V}}' - {V}']^+ \right)  \\
&& \times \left(-E \mathcal{L}_R  {V}'-EK^P V' -E K^I z' \right)\\
& {}={} & -\left( [{V}' - {\overline{V}}']^+ - [{\underline{V}}' - {V}']^+ \right)^T \mathcal{L}_R  {V}' \\
&& - \left( [{V}' - {\overline{V}}']^+ - [{\underline{V}}' - {V}']^+ \right)^T K^P  {V}' \\
&& - \left( [{V}' - {\overline{V}}']^+ - [{\underline{V}}' - {V}']^+ \right)^T K^I z' \\
&& +  z'^T K^I \left( [{V}' - {\overline{V}}']^+ - [{\underline{V}}' - {V}']^+ \right) \\
& {}={} & - \sum_{(i,j) \in \mathcal{E}} \Big[ \left( [{V}'_i - {\overline{V}}'_i]^+ - [{\underline{V}}'_i - {V}'_i]^+ \right) \\
&& - \left( [{V}'_j - {\overline{V}}'_j]^+ - [{\underline{V}}'_j - {V}'_j]^+ \right) \Big] \frac{1}{R_{ij}} \left( {V}'_i - {V}'_j  \right) \\
&& - \sum_{i\in \mathcal{V}} \left( [{V}'_i - {\overline{V}}'_i]^+ - [{\underline{V}}'_i - {V}'_i]^+ \right) K^P_i {V}'_i \\
& {}\le {} & - \sum_{i\in \mathcal{V}} \left( [{V}'_i - {\overline{V}}'_i]^+ - [{\underline{V}}'_i - {V}'_i]^+ \right) K^P_i {V}'_i \le 0.
\end{IEEEeqnarray*}
}
Here function $sgn(x_1, \ldots, x_n)$ is defined as $sgn(x)=\diag(sgn(x_1), \ldots, sgn(x_n))$. The second last inequality follows from the assumptions that $\overline{V}_i = \overline{V}^*$ and $\underline{V}_i = \underline{V}^*$ for $i \in \mathcal{V}$. For the case when ${V}'_i \ge {\overline{V}}'_i$ and ${V}_j \ge {\overline{V}}'_j$, $[{V}'_i - {\overline{V}}'_i]^+ \ge [{V}_j - {\overline{V}}'_j]^+$ implies that  ${V}'_i\ge {V}'_j$. The other cases follow by similar arguments. Thus, the dynamics \eqref{eq:hvdc_dec_voltage_control_cl_hat} converge to a set where ${\underline{V}}' \le {V}' \le {\overline{V}}'$, or equivalently $\underline{V} \le V \le \overline{V}$. As the last inequality coincides with Objective~\ref{obj:voltage_range}, this concludes the proof. 
\end{proof}
The fact that the equilibrium of the closed-loop dynamics \eqref{eq:hvdc_dec_voltage_control_cl} is asymptotically stable regardless of the controller gains facilitates the tuning of the controller \eqref{eq:dec_voltage_control}. Since any controller parameter results in an asymptotically stable equilibrium, the controller can be tuned without stability in mind. 
We now show that under some conditions, the MTDC system \eqref{eq:hvdc_scalar} controlled with the decentralized deadband PI controller \eqref{eq:dec_voltage_control} satisfies Objective~\ref{obj:current_sharing_opt}. 

\begin{lemma}
\label{th:current_sharing_PI_deadband}
If $\norm{I^\text{inj}}$ is sufficiently small and the HVDC line resistances go to zero, the MTDC system \eqref{eq:hvdc_scalar} controlled with the decentralized deadband PI controller \eqref{eq:dec_voltage_control} satisfies Objective~\ref{obj:current_sharing_opt} if the initial voltage $V(0)$ is very close to $V^{nom}$.
\end{lemma}

\begin{proof}
Consider the closed-loop dynamics \eqref{eq:hvdc_dec_voltage_control_cl}. By the assumption that $V^{\text{nom}}_i = V^{\text{nom}}_j, \; i,j \in \mathcal{V}$, $V=V^\text{nom}$ is an equilibrium of \eqref{eq:hvdc_dec_voltage_control_cl} if $I^\text{inj}=0_n$. By continuity of solutions with respect to initial conditions, for sufficiently small $\norm{I^\text{inj}}$, the solutions of \eqref{eq:hvdc_dec_voltage_control_cl} will satisfy $\underline{V}\le V(t) \le \overline{V}\; \forall t\ge 0$ due to stability of \eqref{eq:hvdc_dec_voltage_control_cl} and by Assumption~\ref{ass:nom_voltage}. Thus, $z(t)=0\; \forall t\ge 0$, and the dynamics \eqref{eq:hvdc_dec_voltage_control_cl} reduce to those corresponding to a decentralized droop controller, \eqref{eq:hvdc_closed_loop_droop_vector}. By Lemma~\ref{lemma:droop_current_sharing_resistance}, Objective~\ref{obj:current_sharing_opt} is satisfied in the limit when $R_{ij}\rightarrow 0, (i,j) \in \mathcal{E}$. This concludes the proof.
\end{proof}
We have now shown that for sufficiently small disturbances $I^\text{inj}$ and resistances $R_{ij}$, an MTDC system \eqref{eq:hvdc_scalar} controlled with the decentralized deadband PI controller \eqref{eq:dec_voltage_control} satisfies Objective~\ref{obj:current_sharing_opt}. When the disturbance becomes larger, Objective~\ref{obj:current_sharing_opt} is no longer satisfied in general. However, Objective~\ref{obj:voltage_range} is satisfied regardless of the disturbance $I^\text{inj}$. Thus, the HVDC voltages always converge to a value within the operating range. Future work will focus on characterizing Objective~\ref{obj:current_sharing_opt} as a soft constraint by the cost function \eqref{eq:opt1}. By regarding Objective~\ref{obj:voltage_range} as a hard constraint, this constraint can be added to the optimization problem \eqref{eq:opt1} of Objective~\ref{obj:current_sharing_opt}. Such a formulation is however likely to require a distributed controller structure, allowing for communication between HVDC terminals. An important question is how Objectives~\ref{obj:voltage_range} and \ref{obj:current_sharing_opt} can be combined for controllers with decentralized structure.

\section{Simulations}
\label{sec:simulations}
In this section we demonstrate the proposed decentralized PI controller with deadband \eqref{eq:dec_voltage_control} on the four-terminal MTDC system illustrated in Figure~\ref{fig:graph}. For reference, also the decentralized voltage droop controller \eqref{eq:droop} was simulated. The MTDC system was simulated using \textsc{Matlab}. The HVDC line resistances are assumed to be $3.7$ $\Omega$ for all lines. The lower and upper voltage limits are by $95$ kV and $105$ kV for all terminals, respectively.
\begin{figure}[t]
	\centering
\tikzsetnextfilename{HVDC_decentralized_dynamical_V}
	\begin{tikzpicture}
	\begin{axis}
	[cycle list name=Voltage,
	xlabel={$t$ [ms]},
	ylabel={$V(t)$ [kV]},
	xmin=-0.05,
	xmax=1.5,
	ymin=91,
	ymax=101,
	yticklabel style={/pgf/number format/.cd,
		fixed,
		precision=4},
	grid=major,
	height=4.5cm,
	width=4.5cm,
	legend cell align=left,
	legend pos= north east,
	legend entries={Deadband PI controller, Droop controller, $\underline{V}$
	},
	]
	\addlegendimage{blue,line width=0.5pt}
	\addlegendimage{red,dashed,line width=0.5pt}
	\addlegendimage{dotted,line width=0.5pt}
	\foreach \x in {1, 2, 3, 4, 5, 6, 7, 8, 9}{
	\addplot table[x index=0,y index=\x,col sep=space]{Simulations/HVDC_decentralized_dynamical_V.txt};
	}	
	\end{axis}
	\end{tikzpicture}
\caption{Voltages of the HVDC terminals for the decentralized PI controller with deadband \eqref{eq:dec_voltage_control}, and for the decentralized voltage droop controller \eqref{eq:droop}, respectively. }
	\label{fig:HVDC_decentralized_V}
\end{figure}
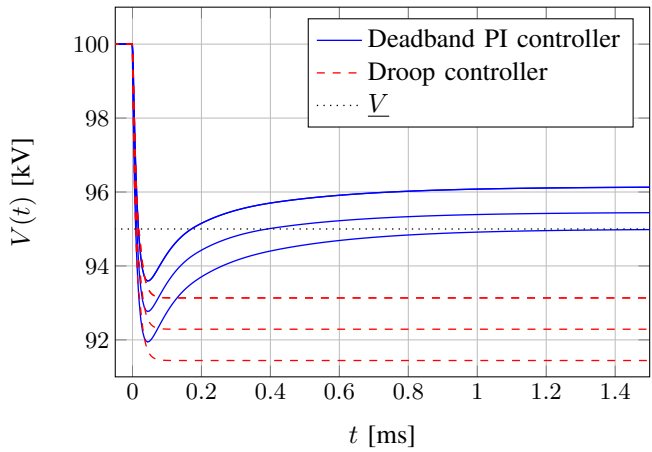
\begin{figure}[t]
	\centering
\tikzsetnextfilename{HVDC_decentralized_dynamical_I_inj}
	\begin{tikzpicture}
	\begin{axis}
	[cycle list name=InjectedCurrent,
	xlabel={$t$ [ms]},
	ylabel={$u(t)$ [A]},
	xmin=-0.05, 
	xmax=1.5,
	ymin=-100,
	ymax=1400,
	yticklabel style={/pgf/number format/.cd,
		fixed,
		precision=4},
	grid=major,
	height=6.5cm,
	width=\columnwidth,
	legend cell align=left,
	legend pos= south east,
	legend entries={Deadband PI controller, Droop controller
	},
	]
	\addlegendimage{blue,line width=0.5pt}
	\addlegendimage{red,dashed,line width=0.5pt}
	\foreach \x in {1, 2, 3, 4, 5, 6, 7, 8}{
	\addplot table[x index=0,y index=\x,col sep=space]{Simulations/HVDC_decentralized_dynamical_I_inj.txt};
	}	
	\end{axis}
	\end{tikzpicture}
\caption{Controlled injected currents for the decentralized PI controller with deadband \eqref{eq:dec_voltage_control}, and for the decentralized voltage droop controller \eqref{eq:droop}, respectively. }
	\label{fig:HVDC_decentralized_u}
\end{figure}
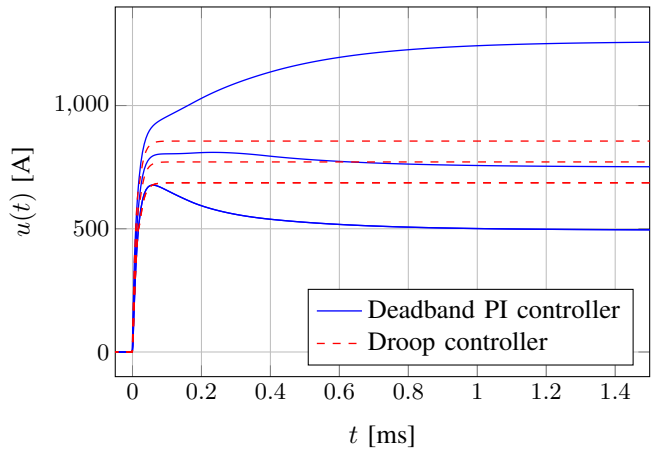
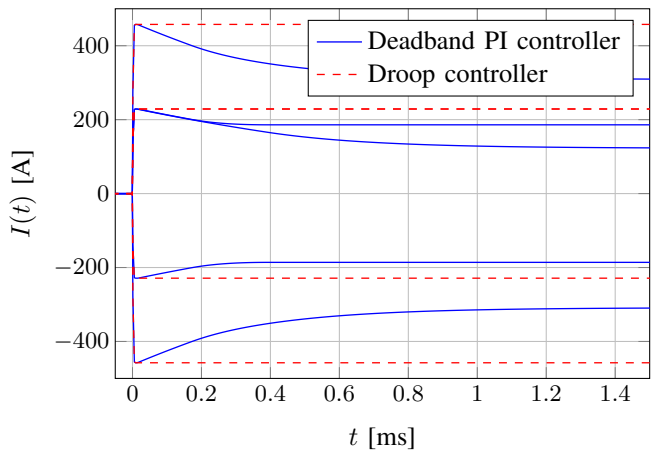
\begin{figure}[t!]
	\centering
\tikzsetnextfilename{HVDC_decentralized_dynamical_I_line}
	\begin{tikzpicture}
	\begin{axis}
	[cycle list name=LineCurrent,
	xlabel={$t$ [ms]},
	ylabel={$I(t)$ [A]},
	xmin=-0.05,
	xmax=1.5,
	ymin=-500,
	ymax=500,
	yticklabel style={/pgf/number format/.cd,
		fixed,
		precision=4},
	grid=major,
	height=6.5cm,
	width=\columnwidth,
	legend cell align=left,
	legend pos= north east,
	legend entries={Deadband PI controller, Droop controller
	},
	]
	\addlegendimage{blue,line width=0.5pt}
	\addlegendimage{red,dashed,line width=0.5pt}
	\foreach \x in {1, 2, 3, 4, 5, 6, 7, 8, 9, 10}{
	\addplot table[x index=0,y index=\x,col sep=space]{Simulations/HVDC_decentralized_dynamical_I_line.txt};
	}	
	\end{axis}
	\end{tikzpicture}
\caption{HVDC line currents for the decentralized PI controller with deadband \eqref{eq:dec_voltage_control}, and for the decentralized voltage droop controller \eqref{eq:droop}, respectively. }
	\label{fig:HVDC_decentralized_I_line}
\end{figure}
The controller parameters are given by $k^V=10$ and $k^I=10$. The simulation was initialized at the voltages $V=[105, 104.96, 104.9, 105.07]^T$ kV. This implies that the voltage of the fourth HVDC terminal is initially larger than the maximum voltage. The simulation results are illustrated in Figures~\ref{fig:HVDC_decentralized_V}--\ref{fig:HVDC_decentralized_I_line} which show the DC voltages, the controlled injected currents and the HVDC line currents for both the decentralized droop controller \eqref{eq:droop} and the deadband PI controller \eqref{eq:dec_voltage_control}, respectively. Figure~\ref{fig:HVDC_decentralized_V} shows that while the droop controller \eqref{eq:droop} fails to satisfy Objective~\ref{obj:voltage_range}, the voltages under the deadband PI controller \eqref{eq:dec_voltage_control} all converge to within the acceptable operating range. Figure~\ref{fig:HVDC_decentralized_u} shows that the power sharing property of the droop controller \eqref{eq:droop} is lost under the deadband PI controller \eqref{eq:dec_voltage_control}, where one HVDC terminal needs to inject a considerable increased amount of current in order to drive the DC voltages within the acceptable operation range. 

\section{Discussion and conclusions}
\label{sec:discussion}
In this paper we have studied decentralized controllers for voltage control of MTDC grids. Starting with the droop controller we have shown that while it does in general not guarantee the voltages to converge within the acceptable operation range, it is able to share the current injections fairly provided that the HVDC line resistances are sufficiently low. We subsequently showed that while a decentralized PI controller guarantees the voltages to converge within the acceptable operation range, the injected currents are in general not distributed fairly. We proposed a decentralized PI controller with deadband to overcome the aforementioned challenges. The proposed controller guarantees the voltages to converge within the acceptable operation range, while providing fair sharing of the injected currents under certain conditions.
A major advantage of the proposed controller is its fully decentralized structure, making it much less vulnerable to communication delays than distributed (with local communication) or centralized controllers. 
 The effectiveness of the proposed controller was validated through simulation on a four-terminal MTDC grid. Future work will focus on unifying the two objectives studied in this paper. By assigning a cost to the disability to fairly share the injected currents the current sharing objective may be relaxed. 

\bibliographystyle{IEEEtran}
\bibliography{references}

\begin{thebibliography}{10}
\providecommand{\url}[1]{#1}
\csname url@samestyle\endcsname
\providecommand{\newblock}{\relax}
\providecommand{\bibinfo}[2]{#2}
\providecommand{\BIBentrySTDinterwordspacing}{\spaceskip=0pt\relax}
\providecommand{\BIBentryALTinterwordstretchfactor}{4}
\providecommand{\BIBentryALTinterwordspacing}{\spaceskip=\fontdimen2\font plus
\BIBentryALTinterwordstretchfactor\fontdimen3\font minus
  \fontdimen4\font\relax}
\providecommand{\BIBforeignlanguage}[2]{{%
\expandafter\ifx\csname l@#1\endcsname\relax
\typeout{** WARNING: IEEEtran.bst: No hyphenation pattern has been}%
\typeout{** loaded for the language `#1'. Using the pattern for}%
\typeout{** the default language instead.}%
\else
\language=\csname l@#1\endcsname
\fi
#2}}
\providecommand{\BIBdecl}{\relax}
\BIBdecl

\bibitem{padiyar1990hvdc}
K.~R. Padiyar, \emph{{HVDC} power transmission systems: technology and system
  interactions}.\hskip 1em plus 0.5em minus 0.4em\relax New Age International,
  1990.

\bibitem{bresesti2007hvdc}
P.~Bresesti, W.~L. Kling, R.~L. Hendriks, and R.~Vailati, ``{HVDC} connection
  of offshore wind farms to the transmission system,'' \emph{IEEE Transactions
  on Energy Conversion}, vol.~22, no.~1, pp. 37--43, 2007.

\bibitem{van2010multi}
D.~Van~Hertem and M.~Ghandhari, ``Multi-terminal~vsc~hvdc for the european
  supergrid: Obstacles,'' \emph{Renewable and Sustainable Energy Reviews},
  vol.~14, no.~9, pp. 3156--3163, 2010.

\bibitem{kundur1994power}
P.~Kundur, \emph{Power System Stability and Control}, ser. The {EPRI} Power
  System Engineering.\hskip 1em plus 0.5em minus 0.4em\relax McGraw-Hill
  Companies, Inc., 1994.

\bibitem{Haileselassie2012}
T.~M. Haileselassie and K.~Uhlen, ``Impact of dc line voltage drops on power
  flow of {MTDC} using droop control,'' \emph{IEEE Transactions on Power
  Systems}, vol.~27, no.~3, pp. 1441--1449, Aug 2012.

\bibitem{haileselassie2009control}
T.~Haileselassie, K.~Uhlen, and T.~Undeland, ``Control of multiterminal {HVDC}
  transmission for offshore wind energy,'' in \emph{Nordic Wind Power
  Conference}, 2009, pp. 10--11.

\bibitem{14}
P.~Karlsson and J.~Svensson, ``{DC} bus voltage control for a distributed power
  system,'' \emph{IEEE Transactions on Power Electronics}, vol.~18, no.~6, pp.
  1405 -- 1412, nov. 2003.

\bibitem{dai2010impact}
J.~Dai, Y.~Phulpin, A.~Sarlette, and D.~Ernst, ``Impact of delays on a
  consensus-based primary frequency control scheme for {AC} systems connected
  by a multi-terminal {HVDC} grid,'' in \emph{IEEE Bulk Power System Dynamics
  and Control}, 2010, pp. 1--9.

\bibitem{andreasson2014MTDCjournal}
M.~Andreasson, D.~V. Dimarogonas, and K.~H. Johansson, ``Distributed
  controllers for multi-terminal {HVDC} transmission systems,''
  \emph{arXiv:1411.1864}, 2014.

\bibitem{Andreasson2014_IFAC}
M.~Andreasson, M.~Nazari, D.~V. Dimarogonas, H.~Sandberg, K.~H. Johansson, and
  M.~Ghandhari, ``Distributed voltage and current control of multi-terminal
  high-voltage direct current transmission systems,'' in \emph{IFAC World
  Congress}, Aug. 2014.

\end{thebibliography}

\end{document}